\documentclass[12pt]{amsart}

\usepackage[T1]{fontenc}
\usepackage{tikz-cd}
\usepackage{amsmath}
\usepackage{calligra}
\usepackage{amssymb}
\usepackage{amsthm}
\usepackage{xcolor}
\usepackage[margin=1.2in]{geometry}
\usepackage{hyperref}
\usepackage{mathrsfs} 
\usepackage{yfonts}
\usepackage{egothic}
\usepackage{aurical}
\newtheorem{thm}{Theorem}
\newtheorem{dfn}{Definition}
\newtheorem{prop}{Proposition}
\newtheorem{cor}{Corollary}
\newtheorem{lem}{Lemma}
\newtheorem{rem}{Remark}
\newtheorem{ex}{Example}
\newtheorem{D-prop}{Definition-Proposition}

\newcommand{\ft}{\frak{t}}

\newcommand{\fg}{\frak{g}}

\DeclareMathAlphabet{\mathcalligra}{T1}{calligra}{m}{n}

\newcommand{\C}{\mathbb{C}}
\newcommand{\R}{\mathbb{R}}
\newcommand{\hH}{\mathbb{H}}
\newcommand{\oO}{\mathbb{O}}
\newcommand{\K}{\mathbb{K}}

\newcommand{\cA}{\mathcal{A}}

\newcommand{\cJ}{\mathcal{J}}
\newcommand{\cH}{\mathcal{H}}
\newcommand{\cM}{\mathcal{M}}
\newcommand{\cO}{\mathcal{O}}
\newcommand{\cT}{\mathcal{T}}

\title{On Frobenius structures in symmetric cones}
\author{Noemie Combe}
 
\begin{document}
\maketitle

\begin{abstract} 
We prove that in any strictly convex symmetric cone $\Omega$ there exists a non empty locus where the WDVV equation is satisfied (i.e. there exists a hyperplane being a Frobenius manifold). This result holds over any real division algebra (with a restriction to the rank 3 case if we consider the field $\oO$) but also on their linear combinations. This theorem holds as well in the case of pseudo-Riemannian geometry, in particular for a Lorentz symmetric cone of Anti-de-Sitter type. 
Our statement can be considered as a generalisation of a result by Ferapontov--Kruglikov--Novikov and Mokhov. Our construction is achieved by merging two different approaches: an algebraic/geometric one and the analytic approach given by Calabi in his investigations on the Monge--Amp\`ere equation for the case of affine hyperspheres. 
\end{abstract}
\smallskip 

\setcounter{tocdepth}{1}
\tableofcontents

\section{Introduction} 
\subsection{The main problem and result}
Consider a strictly convex symmetric cone in the Euclidean space. Does there exist a relation between  strictly convex symmetric cones and the Witten--Dijkgraaf--Verlinde--Verlinde (WDVV) highly nonlinear PDE equation (also known as forming Frobenius manifolds)? The answer to this question is yes. In this paper, we prove that in a strictly convex symmetric cone there exists a non empty locus (a hyperplane) where the WDVV equation is satisfied. This result holds over any division algebra (as well as their linear combinations). In addition, this holds also for symmetric Lorentz cones of an Anti-de-Sitter type. We  separate the latter from the first class of cones due to the difference in the geometries that they carry. The construction of the proof allows a deep exposition on the geometrical and algebraic aspects occurring around those types of manifolds and  answers questions raised by Yu. I. Manin, related to our joint works. 

Our statement can be considered as a generalisation of a construction by Ferapontov--Kruglikov--Novikov~\cite{FKN} and Mokhov~\cite{Mokh95,Mokh09} proving the existence of connections between the symplectic Monge--Amp\`ere equations of Hirota type and the WDVV equation. The construction is achieved by merging two different approaches: an algebraic/geometric one and the analytic approach given by Calabi in his investigations on the Monge--Amp\`ere equation for the case of affine hyperspheres.

For a documentation on the importance of the WDVV equation in algebraic geometry we refer to \cite{Du96,KoMa,KKP08,Man98,Man99,Man05,Sa}. Note that a manifold satisfying the WDVV equation corresponds to a Frobenius manifold for algebraic geometers. For a more differential geometry flavoured approach, relating the WDVV equation and integrable systems see for instance~\cite{FM96,Mokh08,Mokh09}.

\subsection{Motivation}
The reason for our question takes root in a series of differential geometry problems. In particular discussions with M. Kontsevich~\cite{Kont} at the IHES during my stay Nov-Dec (2022) and discussions with  Yu. I. Manin that have followed from our joint works~\cite{CoMa,CoMaMa22A,CoMaMa22B} have inspired this paper. One of the differential geometry problems related to what we consider is the problem of classifying flat Lagrangian submanifolds in $\R^{2n}$ with a given pseudo-Riemannian metric. From the Hessian geometry perspective, having flat Hessian metrics corresponds exactly to satisfying the WDVV equation (see \cite{Kito}). Finally, this problem is also connected to problems gravitating around manifolds of constant curvature (\cite{To04}).

\subsection{The realm of strictly convex symmetric cones: 1935--now}
Strictly convex symmetric cones play a central role in many different domains. 
They appeared at first in the works of Cartan~\cite{Ca}, Koszul~\cite{Kos59,Kos61,Kos62,Kos65,Kos68A,Kos68B} and Vinberg~\cite{Vi60,Vin,Vi65}.
From the works of Minkowski, Siegel~\cite{Sie35,Sie44}, Maas~\cite{Maa}, Piateski--Shapiro~\cite{PS} (and many others) those cones are important in number theory. In algebraic geometry, symmetric convex cones appear under the shape of ''cone of K\"ahler classes''  for the case of an $n$-dimensional complex torus \cite{Wil,To04}. Throughout the works of Wishart~\cite{Wi28}, Constantine \cite{Co63,Co66}, James~\cite{Ja}, Muirhead~\cite{CoMu72,CoMu76}, symmetric cones shine in statistics and in harmonic analysis \cite{FK}. More recent developments show links towards information geometry \cite{AnWo,CoMa,CoMaMa22A,CoMaMa22B}. Their importance is revealed also on a more applied side of mathematics, such as in optimization via convex programming and machine learning. 

\subsection{The rich geometry and algebra of symmetric cones}
Behind the notion of strictly convex  symmetric cones hides a rich geometric and algebraic world. Our proof is the fruit of merging two different languages, which intersect at the notion of those strictly convex symmetric cones. First we use a differential geometry aspect. This comes from the relations between those cones and the Monge--Amp\`ere equation (showed in Lem.~\ref{L:MA}). The more algebraic approach is inherited from Cartan’s classification of symmetric spaces. 

\smallskip
 
Recall the bridge between strictly convex symmetric cones and a special case of the Monge--Ampere equation. For $x\in\R^n$ and $\phi(x)$ a smooth function, this equation is:
\begin{equation}\label{E:MA}
 \det \mathrm{Hess}(\phi)=k,  
\end{equation}
where $k$ is a constant and $\mathrm{Hess}(\phi)$ is the Hessian of $\phi$.

Following \cite{Cal}, Eq. \ref{E:MA} has at most one convex solution in a bounded strictly convex domain, if $\phi$ has prescribed boundary values. Interest is given to tensors defined by the second and third derivatives of $\phi$. Given $\phi$ a smooth function on an open subset of a real vector space, it is possible to define an associated Hessian metric. This metric is obtained by taking the second derivatives of $\phi$. Hessian metrics are a natural way to construct Riemannian or pseudo-Riemannian metrics.

Using Cartan’s classification of symmetric spaces those cones are expressed using a more algebraic  flavour. It turns out that those strictly convex cones fall into two main classes. The first class corresponds to so-called Lagrangian Grassmannians having nonpositive sectional curvature. The second class are formed by the Lorentz cones (of anti-de-Sitter type). 
 
 \subsection{A proof giving a new perspective}
By merging the geometric and algebraic viewpoints discussed above, we create a method which allows to prove our result. This new method gives additionally a very detailed description of the geometric and algebraic landscape occurring in this topic. For instance, using the language of Lie groups, we reformulate one part of our statement by saying that: in a noncompact Lagrangian--Grassmanian symmetric space there exist hypersurfaces satisfying the axioms of a Frobenius manifold. Those spaces are defined over a division algebra $\R, \C, \hH, \oO$ (or a linear combination of those algebra). Reciprocally, there exist hypersurfaces satisfying the axioms of a Frobenius manifold within a Lorentz symmetric cone (of anti-de-Sitter type). 

 \subsection{Connections to other works}  \cite{FKN,Mokh95} shows that in the real case, using a reformulation of the dispersionless Hirota type of equation (in the context of the symplectic Monge--Ampere equation) one can prove the existence of a hyperplane such that the WDVV equation is satisfied. In the real case, this statement is linked to our result since it is a confirmation that our statement is true. Since our result holds over any division algebra (and their linear combinations) as well as for Lorentz Anti-de-Sitter (AdS) symmetric cones, our result gives a general statement. Finally, we point out that the method chosen to prove our result pictures the rich geometric and algebraic patterns, hiding  within those symmetric cones and bridges two different ways of seeing them. 

 \subsection{Organisation of the paper}
This paper is organised as follows: in the three first sections we expose the geometric and algebraic state of the art concerning the strictly convex symmetric cones. It is necessary to recall some known results for the proof. The the last two sections are devoted to the construction the proof of the main statement. 

\thanks{{\bf Acknowledgments} I would like to thank M. Kontsevich for many discussions and comments during my stays (spring and winter 2022) at the IHES. Both Max Planck Institutes for Mathematics in Bonn (MPIM) and in Leipzig (MPIMIS) are very much acknowledged for having supported my research. As well, I am grateful to the grant Polonez-bis 3  for supporting my research.}

\section{Strictly convex symmetric cones}\label{S:Vinberg}

This part serves, as a first focus, on what is known of the geometric aspect of those strictly symmetric convex cones. The second part will bring a (complementary) algebraic version. 
\subsection{Strictly convex cones} 
In the following parts of this article we always consider {\it strictly convex cones}. Note that for brevity we simply refer to them as {\it convex cones.} 

Let us recall some elementary notions on strictly convex cones (see \cite{FK} for further information). 

\begin{dfn}
Let $V$ be a finite dimensional real vector space. Let $\langle-,-\rangle$ be a non-singular symmetric bilinear form on $V$.
A subset $\Omega \subset V$ is a convex cone if and only if $x,y \in \Omega$ and $\lambda,\mu >0$ imply $\lambda x+\mu y \in \Omega$.
\end{dfn} 

\subsection{Homogeneous cones} 
The automorphism group $G(\Omega)$ of an open convex cone $\Omega$ is defined by 
\[G(\Omega)=\{g\in GL(V)\, |\, g\Omega=\Omega\}\]
An element $g\in GL(V)$ belongs to $G(\Omega)$ iff $g\overline{\Omega}=\overline{\Omega}$ [Fauraut Koranyi]
So, $G(\Omega)$ is a closed subgroup of $GL(V)$ and forms a Lie group. 
The cone $\Omega$ is said to be {\it homogeneous} if $G(\Omega)$ acts transitively upon $\Omega$.

\smallskip 
\subsection{Symmetric cones} 
From homogeneous cones one can construct symmetric convex cones. Let us introduce the definition of an open dual cone. An open dual cone $\Omega^*$ of an open convex cone is defined by $\Omega^*=\{y\in V\, |\, \langle x,y \rangle>0,\, \forall\, x\in \overline{\Omega}\setminus 0 \}$. A homogeneous convex cone $\Omega$ is symmetric if $\Omega$ is self-dual i.e. $\Omega^*=\Omega$. Note that if $\Omega$ is homogeneous then so is $\Omega^*$. A symmetric homogeneous cone is called a {\it  Vinberg cone}.    
 
 \begin{rem} 
 If $\Omega$ is a symmetric open cone in $V$,  then $\Omega$ is a symmetric Riemann space.    
\end{rem}

\smallskip

\subsection{Automorphism group}  Let us go back to the automorphism group of $\Omega$. This discussion relies on Prop I.1.8 and Prop. I.1.9 in \cite{FK}.

\smallskip 

Let $\Omega$ be a symmetric cone in $V$.  For any point $a\in \Omega$ the stabilizer of $a$ in $G(\Omega)$ is given by 
\[G_a=\{g\in G(\Omega)\, |\, ga=a\}.\]

By [Prop I.1.8 \cite{FK} ], if $\Omega$ is a proper open homogeneous convex cone then for any $a$ in $\Omega$, $G_a$ is compact. Now, if $H$ is a compact subgroup of $G$ then $H\subset G_a$ for some $a$ in $\Omega$. This means that the groups $G_a$ are all maximal compact subgroups of $G$ and that if $\Omega$ is homogeneous then all these subgroups are isomorphic. 

By [Prop. I.1.9, \cite{FK}], if $\Omega$ is a symmetric cone, there exist points $e$ in $\Omega$ such that $G(\Omega)\cap O(V)\subset G_e$, where $O(V)$ is the orthogonal group of $V$. For every such $e$ one has $G_e=G\cap O(V)$ 

Suppose $\Omega$ is a convex homogeneous domain in $V$. Assume that
\begin{itemize}
   \item[---]   $G(\Omega)$ is the group of all automorphisms;
   \item[---]   $G_e=K(\Omega)$ is the stability subgroup for some point $x_0\in \Omega$;
   \item[---]   $T(\Omega)$ is a maximal connected triangular subgroup of $G(\Omega).$ 
\end{itemize}

\smallskip 

Following [\cite{Vin} Th 1] we have: \[G(\Omega)=K(\Omega)\cdot T(\Omega),\] where $K(\Omega) \cap T(\Omega) = e$ and the group  $T$ acts simply transitively. 

This decomposition on the Lie group side leads naturally to its Lie algebra. Cartan's decomposition for the Lie algebra tells us that $\fg=\textgoth{k} \oplus\ft,$ 

where:

\begin{itemize}
   \item[---]   $\frak{t}$ can be identified with the tangent space of $\Omega$ at $e$. 

   \item[---]   $\textgoth{k}$ is the Lie algebra associated to $K(\Omega)$
\end{itemize}
 and
\[[\ft,\ft]\subset \textgoth{k},\]
\[[\textgoth{k},\ft]\subset \ft.\]

From now assume that $G$ is semi-simple. The Killing bilinear form is thus not degenerate on $\fg$ and the symmetric bilinear form is given by:  $$\langle X,Y\rangle=-Tr(ad X\, ad Y)$$ where $ad\, X(\xi)=[X,\xi]$ and $ad\, Y(\xi) =[Y,\xi]$.

\subsection{$\K-$modules}
Throughout the paper let $\K$ be a finite dimensional real division algebra. By the well known Kervaire--Milnor there exist only four such division algebras which are isomorphic to $\R, \C, \hH,$ and $\oO$ of respective dimensions 1,2, 4 and 8. 
\smallskip 
\begin{itemize}
\item[---] If the finite dimensional real division algebra $\K$ is unitary and commutative then it is isomorphic to $\R$ and $\C$.
\item[---] If the real division algebra is noncommutative but associative then $\K$ is isomorphic to $\hH$. 
\item[---] If $\K$ is non associative but alternative then $\K$ is isomorphic to $\oO$.    
\end{itemize}

We discuss briefly the realisation of the $\K$-modules in terms of linear spaces. If $\K$ is a field of dimension $n$, where $n\in \{1,2,4,8\}$ the (real) realisation of the $\K$-module defines a (real) vector space of dimension $nm$ where $m$ is the dimension of the module~\cite{Roz}, Sec. 2.1.1. Via the  Cayley--Dickson process and starting from $\R$, remark that all normed division algebras can be obtained. This means that from $\R$ we get $\C$; from $\C$ we get $\hH$ and finally from $\hH$ we get $\oO$.
Taking the realisation of $\K$-modules as real vector spaces $V$ then 
it leads to: 

\begin{enumerate}
\item the complexification: $V \to (TV ,J)$, where $ J^2=-Id$. 
\item the symplectification: $V  \to (T^* V,\omega),$ where $\omega$ is a non-degenerate closed form.  
\end{enumerate}

See \cite{Sh87,Sh02} for further developments concerning the geometry of the real realisation of $\K$-modules, namely concerning affinors.


\subsection{Classification of cones}
Any symmetric cone (i.e. homogeneous and self-dual) $\Omega$ is in a unique way isomorphic to the direct product of irreducible symmetric cones $\Omega_i$ (cf. Prop. III.4.5, \cite{FK}). According to Vinberg we have that:
\begin{prop}~\label{P:Vclass}
Each irreducible homogeneous self--dual cone belongs to one
of the following classes:
\vspace{3pt}\begin{table}[ht]
    \centering
    \begin{tabular}{|c|c|c|}
  \hline
Nb & Symbol & Irreducible symmetric cones \\
 \hline
 1. &      $ \Pi_n(\R)$ &  Cone of $n \times n$ positive definite symmetric real matrices. \\
     &  & \\
      
    2.  &      $ \Pi_n(\C)$ &  Cone of $n \times n$ positive definite self-adjoint complex matrices. \\
  & \\
       3.  &    $ \Pi_n(\hH)$ &  Cone of $n \times n$ positive definite self-adjoint quaternionic matrices. \\
           &   & \\
        4. &   $ \Pi_3(\oO)$ & Cone of $3 \times 3$  positive definite self-adjoint octavic matrices. \\
           &   & \\
    5. &    $\Lambda_n$    & Lorentz cone  given by $x_0>\sqrt{\sum_{i=1}^n x_i^2}$ (aka spherical cone). \\
        &   & \\
       \hline
    \end{tabular}
    \caption{Classification of irreducible symmetric cones}
    \label{tab:cones}
\end{table}
\end{prop}

\begin{rem}
We have two remarks. The first is that $\Pi_3(\oO)$ corresponds to the Cayley algebra. 
The second is that  spherical cone $\Lambda_n$ corresponds here to an $n$-dimensional Anti-de-Sitter (AdS) space.
\end{rem}

 \subsection{Jordan algebra structures}\label{S:JordanList}
Recall the tight relations between those cones and algebraic objects, namely formally real simple Jordan algebras. 
We introduce some notations: 
\begin{itemize}
   \item[---]   $Sym(n,\mathbb{K})$ denotes the space of symmetric matrices of dimension $n\times n$ defined over the field $\mathbb{K}$.

   \item[---]   $Herm(n,\mathbb{K})$ denotes the space of hermitian matrices of dimension $n\times n$ defined over the field $\mathbb{K}$.
\end{itemize}

\begin{lem}\label{L:tan}
Consider the cones given by $\Pi_n(\K)$, where the field $\K$ is $\R, \C,\hH$ or $\oO$. Then, the tangent space $T_e(\Pi_n(\K))$ to $\Pi_n(\K)$ is the space of self-adjoint matrices.
\end{lem}

\begin{proof}
Take the symmetric cone $\Omega$ of symmetric positive definite matrices $n\times n$ over $\R$. The tangent space to $\Omega$ at $x\in V$ is the space of symmetric matrices with real entries.  The matrix exponential and logarithm realize a one-to-one mapping, between the space
of symmetric matrices to the space of symmetric positive definite matrices. The same reasoning can be made for the cones over the other division algebras.
\end{proof}

 \begin{dfn}
 An algebra $(\mathscr{A}^+,\circ)$ is a Jordan algebra if:
 \begin{itemize} 
   \item[---]   it is commutative and
   \item[---]   $(x^2\circ y)\circ x = x^2\circ (y\circ x)$, 
 \end{itemize} where we have $x^2=x\circ x$ and $x,y\in \mathscr{A}^+.$ 
 
 A Jordan algebra is called Euclidean (or formally real) if it satisfies the formal reality axiom: \[x_1^2+\cdots+x_n^2 =0\quad \text{implies}\quad x_1=\cdots=x_n=0\] 
 \end{dfn}
 
In particular, a formally real Jordan algebra is semi-simple (\cite{Ko} Cor. 5 p.118). To any Vinberg cone there exists a bijectively corresponding semisimple formally real Jordan algebra. 

\smallskip 

\begin{ex}
  By Lem.~\ref{L:tan} the tangent space to  $\Pi_n(\R)$, is $Sym(n,\R)$. Given $X,Y\in Sym(n,\R)$, one gets a new product (giving a Jordan algebra): \[X\circ Y=\frac{1}{2}(XY+YX),\] where $XY$ is the standard matrix product.  
\end{ex}

\begin{prop}
Every formally real Jordan algebra can be written as a direct sum of simple ones. In finite dimensions, the simple formally real Jordan algebras come in four infinite families, together with one exceptional case:

\begin{table}[ht]
    \centering
    \begin{tabular}{|c|c|}
     \hline
 Irreducible symmetric cone & Formally real simpl Jordan algebras  \\
 \hline
$ \Pi_n(\R)$  &  Jordan algebra of $n\times n$ self-adjoint real matrices.  \\
& \\
    $ \Pi_n(\C)$      &  Jordan algebra of $n\times n$  self-adjoint complex matrices.\\
    & \\
    
    $ \Pi_n(\hH)$     &   Jordan algebra of $n\times n$  self-adjoint quaternionic matrices.  \\
    & \\
    
     $ \Pi_3(\oO)$    &  Jordan algebra of $3\times 3$ self-adjoint octonionic matrices:\\
       &  Albert algebra.\\
     $\Lambda_n$  & Spin factor algebra $JSpin^+$
       on the space $\R 1\oplus \R^n$ for $n \geq 2$. \\ 
       & \\
      \hline
    \end{tabular}
    \caption{Classification of symmetric cones using their Jordan algebras}
    \label{tab:Jordan}
\end{table}
\end{prop}

\begin{rem}
Note that one can consider the algebra $JSpin_n^+ $ as a Jordan subalgbera of a full algebra of hermitian matrices $2^n\times 2^n$ real matrices. 
    
\end{rem}

\subsection{Cartan's symmetric spaces} 
Vinberg cones are Cartan symmetric spaces. In Sec.~\ref{S:HessianMfd} we study their properties as Riemannian symmetric spaces. Recall that a (pseudo-)Riemannian manifold is globally symmetric if one can assign to every point $p\in \cM$ an isometry $s_p$ of $\cM$ such that $s^2_p=id$ and $p$ is an isolated fixed point of $s_p$.

\smallskip

A Riemannian symmetric space $\mathcal{M}$ is diffeomorphic to a homogeneous space $G/K$, where $G$ is a connected Lie group with an involutive automorphism whose fixed point set is essentially the compact subgroup $K\subset G.$ Consider the pair $(G,K)$. We call $(G,K)$ a symmetric pair provided that there exists an involution $s\in G$, such that  $(K_{s})_{0} \subset K \subset K_{s}$, where $K_{s}$ is the set of fixed points of $s$ and $(K_{s})_{0}$ is the identity component of $K_{s}$. Riemannian symmetric spaces may be classified in terms of symmetric Lie algebras. Every noncompact symmetric space has a compact dual (and reciprocally).

Any simply connected Riemannian symmetric space is a Riemannian product of irreducible ones. Irreducible, simply connected Riemannian symmetric spaces are classified as follows:

\begin{itemize}

        \item[---]  Type of irreducible symmetric space: Euclidean. The curvature is 0. It is therefore isometric to a Euclidean space.

    \item[---]  Type of irreducible symmetric space: Compact. The sectional curvature is nonnegative (but not identically zero).

      \item[---]   Type of irreducible symmetric space: non-compact.The sectional curvature is nonpositive (but not identically zero). 
         
   \end{itemize}

\begin{dfn}
\begin{enumerate}
\item The following  non-compact symmetric spaces are called spacelike Lagrangian Grassmannians non-compact symmetric spaces:

\begin{table}[ht]
    \centering
    \begin{tabular}{|c|cccc|}
    \hline 
   $\mathbb{A}/   \mathbb{B}$ & $\mathbb{R}$ & $\mathbb{C}$& $\mathbb{H}$ & $\mathbb{O}$ \\
   \hline 
   $\mathbb{C}$& $\frac{SL_n(\R)}{SO(n)} $&  $\frac{SL_n(\C)}{SU(n)}$ &  $ \frac{SL_n(\hH)}{Sp(n)}$ &   $\frac{E_{6(-26})}{F_4}$\\ \hline 
    
 \end{tabular}
    \caption{Spacelike Lagrangian Grassmannian symmetric spaces}
    \label{tab:sym}
\end{table}

\item We call a space of de Sitter type a submanifold of the Minkowski space $\R^{n,1}$, with metric $ds^2=dx_0^2-\sum_{i=1}^{n-1}dx_i^2$ such that the isometry group is a Lorentz group $O(1,n)$. The scalar curvature is positive. 
\item  A space of Anti de Sitter type is a submanifold as in (2) with negative scalar curvature . 
\end{enumerate}
\end{dfn}

\begin{prop}\label{C:VinFrob}
\,

--- Let $\K$ be a real division algebra. The irreducible Vinberg cones $\Pi_n(\K)$ are noncompact symmetric spaces, which are of spacelike Lagrangian Grassmannian type.

--- The irreducible Vinberg cone $\Lambda_n$ is a pseudo-Riemannian noncompact symmetric space, which is a Lorentz manifold (of Anti-De-Sitter type). 

\end{prop}

\begin{proof}
1. Thm.~\ref{P:Vclass}, implies that the first four Vinberg cones are identified to spaces of symmetric positive definite matrices $n\times n$ over $\mathbb{K}$, where $\mathbb{K}$ is  $\R,\C,\hH $ or $\oO$ (concerning $\oO$ the cone holds only for $n=3$). Those cones are symmetric spaces, so diffeomorphic to a quotient of Lie groups $G/K$ such that:

---  $G$ is a connected Lie group

---  $K$ is a Lie subgroup which is (a connected component of) the invariant group of an involution of $G$.

\smallskip

For the class of Vinberg cones $\Omega=\Pi_n(\K)$, one has the following associated Lie groups:

\begin{itemize}
\item[---] $G$ is identified to $GL_n(\mathbb{K})$, where $\mathbb{K}$ is $\R$ (resp. $\C,\hH, \oO$). 
\item[---] The subgroup $K$ is identified to $SO_n$ (resp. $SU_n $, $Sp_n$, where $Sp_n=SU_n(\hH)$, and the exceptional Lie group $F_4$, whenever $\K$ is equal to $\oO$).
\end{itemize}

2. There exists a short exact sequence: 
\[1\to SL_n(\K)\to GL_n(\K)\xrightarrow[]{det} \K^{\times}\to 1, \]

implying that we have the following decomposition $GL_n(\mathbb{K})=SL_n(\K)\rtimes \K^{\times}$, where $\K^{\times}$ is the multiplicative group. One may thus consider $SL_n(\K)/K\rtimes\K^{\times}$ and for simplicity we focus on the submanifold $SL_n(\K)/K$.

Those submanifolds are classified as follows: 
\begin{itemize}
\item[---] $\Pi_n(\C)$, with $\det=1$ is associated to $SL_n(\C)/SU_n$;
\item[---] $\Pi_n(\hH)$, with $\det=1$ is associated  to $SL_n(\hH)/Sp_n;$
\item[---] $\Pi_3(\oO)$, with $\det=1$ is associated to $SL_3(\oO)/F_4$.
\end{itemize}
Concerning the latter, $\Pi_3(\oO)$ note that $Herm(3,\oO)$ is a closed subgroup of the orthogonal group $O(27)$ and thus forms a compact Lie group. It is referred to as $F_4$. 

\smallskip 

3. The spherical cone $\Lambda_n$ is an anti-de-Sitter space. In terms of Lorentzian symmetric spaces it is associated to: $O(1,n-1)/O(n-1)\oplus\R$ (see \cite{FK} p.7 for a detailed proof).

\smallskip 

4. The full classification of Vinberg cones in terms of Lie groups and Jordan algebras is presented in the table below (see for details \cite{FK}, Ch.V, p.97 Sec.3). Te Jordan algebras are listed in the leftmost part of the table. 

\medskip 

\begin{tabular}[ht]{|c|c|c|c|c|c|c|}
 \hline
$\cJ$ & $\Omega$ & $\frak{g}$ & $\frak{k}$ & $dim \cJ $ & rank $\cJ $& $d$  \\
 \hline
  $Sym(n,\R)$ & $\Pi_n(\R)$ & $ \frak{sl}(n,\R)\oplus \R$ & $\frak{o}$(n)&$\frac{1}{2}n(n+1)$ & $ n$ & 1 \\
  $Herm(n,\C)$  &  $\Pi_n(\C)$ & $ \frak{sl}(n,\C)\oplus \R$ & $\frak{su}$(n)& $n^2$ & $ n$ & 2   \\
 $Herm(n,\hH)$ & $\Pi_n(\hH)$ & $\frak{sl}(m,\hH)\oplus \R$ & $\frak{su}(n,\hH$)& $n(2n-1)$ & $n$ & 4 \\
  $ \R\times \R^{n-1}$& $\Lambda_n$  & $ \frak{o}(1,n-1)\oplus\R$ & $\frak{o}$(n-1)& $n$ & 2 & $n-2$ \\
 $ Herm(3,\oO)$ & $\Pi_3(\oO)$ & $ \frak{e}_{(-26)}\oplus \R$ & $\frak{f}_4$ & 27 & 3 & 8  \\
  \hline
\end{tabular}

\medskip 

5. By Nomizu~\cite{No54}, an irreducible symmetric space $G/H$ is either flat, compact or noncompact. 
It is easy to check that the symmetric spaces considered above have non-positive scalar curvature. 
To sumarize:
\begin{itemize}
\item[---] Cones given by $GL_n(\K)/K$ are noncompact symmetric spaces.
\item[---] The spherical cone $\Lambda_n$ corresponds to $O(1,n-1)/O(n-1)\oplus\R$ and belongs also to the class of noncompact symmetric space. 
\end{itemize}

To conclude, Vinberg cones (1)--(4)  correspond to noncompact symmetric spaces of spacelike Lagrangian--Grassmanian type; Vinberg cones  of type (5) are Lorentz cones of Anti-de-Sitter type.
\end{proof}

\medskip 

\begin{prop}[]
Vinberg cones of space-like Lagrangian--Grassmannian type come equipped with a $G$-invariant metric and with a symmetric bilinear form given by $$\langle X,Y\rangle=\Re\, Tr(XY),$$ where $X,Y\in T_e(Gl_n(\K)/K)\cong \ft\subset \fg$ and where $Tr(\cdot)$ stands for the trace (linear) operator .
\end{prop}
\begin{proof}
To the Lie group $Sym(n,\R)$ the corresponding Lie algebra is given by the set of symmetric matrices of trace 0, denoted $\frak{s}ym_0(n)$.This statement follows from the Killing form.
\end{proof}

\section{Hessian structures}\label{S:RHM}
We proceed to a differential geometric approach of cones. This allows to inherit a rich panel of tools, important for our construction.
\subsection{The Monge--Ampere equation}
Consider the non-linear PDE equation Monge--Ampere equation.
\begin{equation}\label{E:MA1}
 \det \mathrm{Hess}(\phi)=f(x), 
 \end{equation}
where  $\mathrm{Hess}(\phi)$ denotes the Hessian matrix attached to the smooth function $\phi$ and $f(x)$ is a given positive-valued function with $x=(x_1,\cdots, x_n)\in \R^n$.

\smallskip 

A special case of the Equ.~\ref{E:MA1} can be given by:
\begin{equation}\label{E:MA2}
 \det \mathrm{Hess}(\phi)=1.  
\end{equation}

Solutions to this nonlinear PDE equation (such that the hessian matrix is definite at each point) are so-called {\it elliptic.} Note that $\phi$ is locally either convex or concave. 

\smallskip 

\begin{lem}\label{L:MA}
In a strictly convex symmetric cone $\Omega$, the equation Eq.~\ref{E:MA1} has at most one convex solution $\phi$ with arbitrary smooth boundary values.
\end{lem}

\begin{proof}
It has been shown that Eq.\ref{E:MA1} has at most one convex solution $\phi$ (with arbitrary smooth boundary values) in any strictly convex bounded domain \cite{Cal}. Given that the strictly convex symmetric cone is also a strictly convex bounded domain. This concludes the statement.  
\end{proof}

Eq. \ref{E:MA} is invariant under (unimodular) linear transformation of the independent variables $x=(x_1,\cdots, x_n)\in \R^n$. If we allow only linear transformations of $x$, the array of all partial derivatives of $\phi$ of any given order $k$ can be interpreted as the components of a covariant tensor of valence $k$ which is symmetric in all pairs of indices. 

Our interest goes to tensors defined by the second and third derivatives. Since we assume that $\phi$ is convex and satisfies \ref{E:MA1}, then the symmetric tensor $g_{ij}$ is positive definite. It thus definesa Riemannian and also Hessian metric, in the domain of definition of $\phi$ . 

\subsection{The potential function}
We introduce a function $\chi(x):\overline{\Omega}\to \R$ such that:
\begin{enumerate}
    \item $\chi(x)$ is real analytic and positive on $\Omega$
     \item $\chi(x)$  is continuous on $\overline{\Omega}$ and vanishes on the boundary $\partial \Omega$.
     \item $\chi(\lambda x)=\lambda^n\chi(x)$ for $\lambda>0$, $x\in \Omega$ and where $n$ is the dimension of $\Omega$.
     \item the bilinear form on $\Omega$ is non singular for every $y\in \Omega$. 
   Fixing a point $c\in \Omega$ we have: 
     \[\langle u,v\rangle=-\partial_u\partial_v\log\chi(y)|_{y=c}\]
    \end{enumerate}

\begin{dfn}[Koszul--Vinberg (KV) characteristic function]\label{D:KV}
Let $\Omega \subset V$ be a strictly convex homogeneous cone. For any vector $x\in \Omega$, define the KV-characteristic function:
\begin{equation}
\chi(x)=\int_{\Omega^*}\exp{\{-\langle x,a^*\rangle\}} da^*
\end{equation} 
where $da^*$ is a volume form invariant under translations in $\Omega^*$.
\end{dfn}
Some data about the KV-function is expressed below. 
\begin{itemize}
\item[---] The KV-function tends to infinity on the boundary of the cone $\overline{\Omega}$. By H\"older's inequality, we have that $\Phi=\ln\chi$ is strictly convex.  
\item[---] From the definition of $\chi$ we have that for any $x\in \Omega$ and any $g\in G(\Omega),$ one has $\chi(gx)=|\det g|^{-1}\chi(x).$ The differential form, $\alpha=d\chi/\chi$ is invariant under $G(\Omega)$. 
\item[---]  Let $\Omega$ be homogeneous cone.  Then the complex tube $T_V$---whose real part is $\Omega$---has a transitive group of holomorphic mappings given by $z\to gz+\imath t,\, g\in G(\Omega)$. The Bergman kernel is for $T_V$ up to constant factor: $\chi^2(z+\overline{w})$.
\end{itemize}

\subsection{Affine flat structures}\label{S:AffineFlat} 
Let $\Omega$ be a strictly convex homogeneous cone. Let $\cT_\Omega$ be the tangent sheaf. Let $\Omega^1_\Omega$ be the sheaf of holomorphic 1-forms on $\Omega$. Let us recall the general definition of an affine flat structure on $\Omega$. It is given by any of the following equivalent data:

\begin{enumerate}
\item An atlas on $\Omega$ whose transition functions are affine linear.

\item A torsionless flat connection $\nabla_0:\cT_\Omega\to \Omega^1_\Omega\otimes_v\cT_\Omega$.

\item A local system $\cT_\Omega^f\subset \cT_\Omega$ of flat vector fields, which forms a sheaf of commutative Lie algebras of rank $n(=dim \Omega)$ such that $\cT_\Omega=\cO_\Omega\otimes_k \cT_\Omega^f$. 
\end{enumerate}

\smallskip 
Affine flat structures on a homogeneous cone imply the existence of a pre-Lie algebra (aka Vinberg algebra, or Left Symmetric Algebra). Consider $X,Y$ two vector fields in $\cT_\Omega$ (or $\frak{t}$). If $\Omega$ is provided with an affine structure then:
\begin{equation}
\nabla_X(Y)-\nabla_Y(X)-[X,Y]=0,
\end{equation}
Also, written as:
\[\nabla_X\nabla_Y-\nabla_Y\nabla_X-\nabla_{[X,Y]}=0,\]
where $X,Y\in \cT_\Omega$.

A multiplication operation $``\circ"$  given by $X\circ Y=\nabla_X(Y)$ defines a commutative algebra $(\frak{t},\circ)$ satisfying the relation:

\[a\circ (b\circ c)- b\circ (a\circ c) = (a\circ b)\circ c- (b\circ a)\circ c,\]

for $a,b,c\in \frak{t}.$

This forms a pre-Lie algebra structure, on the set of connections. This is also known as a Lie-admissible algebra. Each Lie algebra with affine structure is derived from a Lie-admissible algebra. 

\begin{lem}[\cite{Bu}]\label{L:Burde}
There is a one-to-one correspondence of $n$-dimensional convex homogeneous cones $\Omega$ and $n$-dimensional pre-Lie algebras.
\end{lem}

We can thus now state that: 

\begin{cor}\label{L:affst}
Convex homogeneous cones and Vinberg cones are domains equipped with an affine flat structure.
\end{cor}
\begin{proof}
The proof follows direclty from the lemma \ref{L:Burde} and the discussion above.
\end{proof}

\subsection{On Hessian structures}\label{S:HessianStr}
 Suppose that $\Omega$ is an affine flat manifold. Then, it comes equipped with an affine flat connection $\nabla_0$ and there exists a class of Riemannian metrics compatible with $\nabla_0$. A Riemannian metric $g$ on $\mathcal{M}$ is said to be a Hessian metric if $g$ is given by $g=\nabla_0^2 \Phi$, where $\Phi$ is a local smooth function. 

\begin{dfn}
The pair $(\nabla_0 ,g)$ is called a Hessian structure on $\Omega$. The triple given by $ (\Omega, \nabla_0, g) $ is called a Hessian manifold.
\end{dfn}

Let $\Omega$ be a differentiable manifold with a locally flat linear connection $\nabla_0$. Let $p$ be a point of the differentiable manifold $\Omega$ and let $U$ be an open neighborhood of $p$.
Then, for any point $p$ in $U$, there exists a local coordinate system $(x_1 ,..., x_n)$ called an affine local coordinate system, such that $ \nabla_0(dx_i)=0$. A Riemannian metric $g$ on the differential manifold $\Omega$ is said to be locally Hessian with respect to $\nabla_0$ if there exists for each point $p\in \Omega$, a real-valued function $\Phi$ of class $C^{\infty}$ on $U$
such that \[g= \frac{\partial^2 \Phi}{\partial x_i \partial x_j}dx_i dx_j,\] where $(x_1 ,..., x_n)$ is an affine local coordinate system around $p$. Then, $(\nabla_0, g)$ is called a {\it locally Hessian }structure on $\Omega$.

\begin{rem}
A variant of Hessian metrics can be used to define a canonical Riemannian metric on any convex domain, using the solution of a Monge--Amp\`ere equation by Loewner--Nirenberg and Cheng--Yau. 
\end{rem}

Therefore, applying the knowledge above, we can conclude that:  
\begin{lem}\label{l:Hess}
The convex homogeneous cone $\Omega$ has a Hessian structure. 
\end{lem}
\begin{proof}
We use the KV-function to define a Hessian structure on $\Omega$. Indeed, the canonical Riemannian metric attached to the cone $\Omega$ is 
$g_V=g=-\mathrm {Hess}(\ln\chi(x))$. It is invariant under $G(\Omega).$ 

Let us continue the discussion in local coordinates.  Assume $(x_1,\cdots, x_n)$ forms an affine coordinate system on $\Omega$.  The convex homogeneous domain $\Omega$ admits an invariant volume element defined as 
$\Phi dx_1\wedge \cdots \wedge dx_n$. 

\smallskip 
 
The canonical bilinear form is: 
\begin{equation}\label{E:Riem}
g={\color{red}}\sum\frac{\partial^2 \ln \chi}{\partial x_i\partial x_j}dx_idx_j,
\end{equation} 
where $\Phi=\ln\chi(x)$ is a potential function. The  canonical bilinear form $g$ is positively definite. This gives the Riemannian metric on $\Omega$ and defines the Hessian structure.
\end{proof}

This marks the beginning of the emergence of Frobenius manifold patterns. As we will see, $\Phi=\ln\chi(x)$
 is the potential function in the Frobenius manifold (see Sec.~\ref{S:Frobeniusmfd}).

\section{Algebraic structures}\label{S:Algconn}

Let $\Omega$ be a Riemannian homogeneous convex cone (cf. Sec.\ref{S:RHM}). We investigate its relations to algebras.
\subsection{Frobenius algebras}

Consider a finite $n$-dimensional commutative associative algebra $\mathscr{A}$ over $\K$ possibly with unit ${\bf 1}_\mathscr{A}$, with basis $\{e_i\}_{i=1}^n$ and structural constants $C^i_{jk}$, being components of the (1,2)-tensor\, $\circ: \mathscr{A} \times \mathscr{A} \to \mathscr{A}$ such that: 

\[e_i\circ e_j=C_{ij}^s e_s,\quad i,j,s\in \{1,\cdots, n\}.\]

Associativity can be written as follows :
\[ C_{ab}^cC_{cd}^f=C_{ad}^cC_{cb}^f.\]
Commutativity implies that \[C^s_{ab}=C^s_{ba}.\]

\begin{dfn}
 A unital, commutative, associative algebra $(\mathscr{A},\circ)$ equipped with a bilinear symmetric form $\langle -,- \rangle$ satisfying 
\[\langle x\circ y,z \rangle=\langle x,y\circ z \rangle, \quad x,y,z\in \mathscr{A}\] is a Frobenius algebra. 
 \end{dfn}

\subsection{Connection algebras}
\begin{dfn}[Connection algebra]
A connection algebra is a unital, commutative algebra $\mathscr{A}^+$ with a basis $e_1,\cdots, e_n$ and  multiplication $\circ:V\times V \to V$ given by $$e_i\circ e_j=\Gamma_{ij}^ke_k,$$ where $\Gamma_{ij}^k$ are the structure constants and being equipped with a symmetric bilinear form $\langle -,-\rangle$, induced from the Riemannian metric $g$. 
\end{dfn}

\smallskip 

\begin{D-prop}\label{P:FA}
Let $\Omega$ be a convex (pseudo-)Riemannian homogeneous cone;
let $p\in \Omega$ be a point. Consider a system of linear coordinates $(x_1,x_2,\cdots, x_n)\in \Omega$. 

Then, there exists a connection algebra $\mathscr{A}^+$ on the tangent bundle of $\Omega$ with multiplication operation $\circ$ given for every pair $X,Y$ in $V$ by:

\begin{equation}\label{E:circ}
(X\circ Y)^i=-\sum_{j,k}\Gamma^{i}_{jk}X^jY^k\quad 1\leq i\leq n.\end{equation}

Structure constants are given by $\Gamma^{i}_{jk}=\frac{1}{2}\partial_{jkl}\Phi g^{li}(p)$ and
$\partial_{jkl}\Phi=C(\partial_j,\partial_k,\partial_l)$ forms a rank three symmetric tensor. 
\end{D-prop}

The construction of this algebra follows from~\cite{Vin}. 

\begin{rem}
Note that the symmetry of $C$ comes from the fact that $\chi(e)$ is a real-analytic function. 
\end{rem}
\subsection{Towards Jordan algebras}
The connection algebra $\mathscr{A}^+$ forms a Jordan algebra if the Riemannian space is a symmetric space (i.e. is a Vinberg cone). We define the powers of the elements $u$ in $\mathscr{A}^+$ as follows, which leads to the Jordan algebra: 

\[X^1=X,\quad X^{m+1}:=X\circ X^{m},\quad m\geq 1.\]
\begin{cor}
      If $\mathscr{A}^+$ is a Jordan algebra, then $X^r\circ X^s=X^{r+s}$ holds for all
$X \in  \mathscr{A}^+$, where $r\geq 1$ and $s\geq 1$ and it is then called power-associative.
\end{cor}

The Jordan algebras associated to the Vinberg cones are  finite-dimensional formally real Jordan algebras, being  a direct sum of a finite number of simple ideals. The  five basic types of simple building blocks are listed in Sec.\ref{S:JordanList}

\begin{lem}
 The pre-Lie algebra $\nabla_X(Y)=X\circ Y$ generates the connection algebra $\mathscr{A}^+$.
\end{lem}

\begin{proof}
The pre-Lie algebra defined above is given by a multiplication operation  $X\circ Y=\nabla_X(Y) $ where $X,Y$ lie in the space of vector fields of $\frak{t}$.
The algebra $\mathscr{A}^+$ is defined on the vector space $V$ which identifies with $\frak{t}$.

Let us construct a canonical linear isomorphism between $\frak{t}$ and $V$, via Vinberg's $T$-algebras \cite{Vin}. Every element $a\in \frak{t}$ corresponds to the sum of $a + \tilde{a}\in V=T_e(V)$, where $\tilde{a}$ is the involutive anti-automorphism equipping the $T$-algebra. This determines the isomorphism: 

\[\nu: \frak{t}\to  V=T_e(V)\]
\[\, a\mapsto \, a+\tilde{a}. \]

The multiplication operation on $\mathscr{A}^+$ coincides with $X\circ Y=\nabla_X(Y)$. So, the pre-Lie algebra $\nabla_X(Y)=X\circ Y$ generates the connection algebra $\mathscr{A}^+$.
\end{proof}

\smallskip 

The algebra is defined throughout a non-singular symmetric bilinear form on $V$. Consider a trilinear form $C$ on $V$. Fixing one argument we get a bilinear form in the remaining two arguments. Therefore, for given $X$ the form $C(X,Y,Z)$ is a bilinear form in $Y$ and $Z$ and there exists a linear transformation $\nabla_X$ such that \[C(X,Y,Z)=g(\nabla_X(Y),Z).\]
If the trilinear form is symmetric in $X,Y,Z$ then we have that the algebra is commutative and that $$g(X\circ Y,Z)=g(X,Y\circ Z).$$

\begin{lem}\label{L:Ass}
 Let $\mathscr{A}^+$ be the connection algebra. The symmetric bilinear form satisfies:   \[ \langle x\circ y, z\rangle=\langle x, y\circ z\rangle,\] for all $x,z,y\in \mathscr{A}^+$.
\end{lem}
\begin{proof}
  The symmetry of the rank 3 tensor $C(x,y,z)$ implies the relation:
\[ \langle x\circ y, z\rangle=\langle x, y\circ z\rangle.\]  
\end{proof}

\begin{cor}[\cite{Ko}, Thm. 12, p.118.]
Let $\mathscr{A}^+$ be a real Jordan algebra associated to a Vinberg cone $\Omega$. Then the following statements are equivalent:
\begin{itemize}
\item $\mathscr{A}^+$ is formally real;
\item there exists a positive definite bilinear form $\langle-,- \rangle$ satisfying 
\[\langle x,y\circ z \rangle=\langle x\circ y,z \rangle.\]
\end{itemize}
\end{cor}
\begin{proof}
 This follows from \cite{Ko}, Thm. 12, p.118.
In fact, in a simple Euclidean Jordan algebra, every associative symmetric bilinear form is a scalar multiple of $Tr(xy)$.
\end{proof}
\begin{lem}\label{L:ConnFrob}
Consider an associative, commutative unital subalgebra $ \mathscr{A} \subset \mathscr{A}^+$. Then, this subalgebra forms a Frobenius algebra. 
\end{lem}

\begin{proof}
A Frobenius algebra is a unital, commutative, associative algebra (finite dimensional) equipped with a symmetric bilinear form $\langle-,-\rangle$ where:  $\langle x\circ y,z\rangle=\langle x,y \circ z\rangle$, forall $x,y,z$ lying in the Forbenius algebra.  

The subalgebra $\mathscr{A}$ is a commutative, associative, unital subalgebra of $\mathscr{A}^+$, equipped with a symmetric bilinear form satisfying associativity condition (by Lem.\ref{L:Ass}). So, $\mathscr{A}$ is a Frobenius algebra.
\end{proof}

\section{Frobenius structures}\label{S:Frobeniusmfd}
We set out some notions on Frobenius manifolds/ structures and conclude by showing that this is equivalent to satisfying the WDVV equation. 

\subsection{Compatible flat affine structures with a multiplication}
Recall the notion of flat structures introduced in Sec.~\ref{S:AffineFlat}. For brevity, let $\cM$ be a smooth manifold. Assume $\cT_\mathcal{M}$ is endowed with  $\circ$ a $\cO_\mathcal{M}$-bilinear commutative and associative  multiplication operation, eventually with unit $e$.  

\smallskip

 Then, according to \cite{Man05} [Def.2.2.1], a flat structure $\cT_\mathcal{M}^f$ on $\mathcal{M}$ is called compatible with $\circ$, if in a neighbourhood of any point there exists a vector field $\mathscr{C}$ such that for any arbitrary local flat vector fields we have 

\begin{equation}\label{E:C}
X\circ Y= [X,[Y,\mathscr{C}]],
\end{equation}
 $\mathscr{C}$ is a local vector potential for $\circ$.

Moreover, $\cT_\mathcal{M}^f$ is called compatible with $(\circ,{\bf 1}_\mathscr{A})$ if the equation \eqref{E:C}  holds and if $e$ is flat.

By \cite{Man05} Prop 2.2.2: If $\circ$ admits a compatible flat structure then it satisfies the identity:
 For all local vector fields $X,Y,Z,W$:
\begin{equation}\label{E:Poisson}
 P_{X\circ Y}(Z,W)=X\circ P_Y(Z,W)+(-1)^{XY}Y\circ P_{X}(Z,W), 
\end{equation}
where $P_{X}(Z,W):=[X,Z\circ W]-[X,Z]\circ W -(-1)^{XZ}Z\circ [X,W]$.

A manifold $(\cM,\circ)$, where $\circ$ is an $\cO_\mathcal{M}$-bilinear commutative and associative  multiplication operation satisfying the identity ~\ref{E:Poisson} is called an $F$-manifold (see Def 5.1 \cite{Man99}.

\subsection{Pencils of flat connections}\label{S:flatnesscondition}

\smallskip 
 
Let us consider the following input data:

\begin{itemize}
\item a flat structure $\nabla_0:\cT_\mathcal{M}\to \Omega^1_\mathcal{M}\otimes_\mathcal{M}\cT_\mathcal{M}$ on $\mathcal{M}$, where $\Omega_\mathcal{M}^1$ is the sheaf of holomorphic 1-forms on $\mathcal{M}$.
\item  An odd global section $\cA\in \Omega^1_\mathcal{M}\otimes_{\cO_\mathcal{M}}End(\cT_\mathcal{M}).$
\end{itemize}
Then, one can produce from it the following datum:

\begin{itemize}
\item A pencil of connections $\nabla_{\lambda}^\cA=\nabla_0+\lambda \cA$. 
\item  An $\cO_\mathcal{M}$-bilinear composition law $\circ$ on $\cT_\mathcal{M}$ given by:
\end{itemize}
\[X\circ Y:= i_X(\cA)(Y),\] 
where for $G\in End(\cT_\mathcal{M})$ and $df\in\Omega_\mathcal{M}^1$ (for $f\in \cO_\mathcal{M}$) the following is defined: $i_X(df\otimes G):= Xf \cdot G.$

\begin{prop}[\cite{Man05}, Prop 2.3.1.]\label{P:F}
$(\cM,\circ,\nabla^\cA_0)$ is an $F$-manifold with compatible flat structure iff $\nabla_{\lambda}^\cA$ is a pencil of torsionless flat connections.   
\end{prop}

In that case, $(\cM,\circ,\nabla_{\lambda}^\cA)$ is an $F$-manifold with compatible flat structure for any $\lambda$ as well. 
\subsection{(pre-)Frobenius manifolds}\label{S:FrobeniusmfdDATA}

 \medskip
 
Consider the following family of data:
\begin{equation}\label{E:1}
(\mathcal{M};\, \circ:\cT_\mathcal{M}\otimes \cT_\mathcal{M}\to \cT_\mathcal{M};\, \cT_\mathcal{M}^f\subset \cT_\mathcal{M}; \ g:S^2(\cT_\mathcal{M})\to \cO_\mathcal{M}),
 \end{equation}
 where: \begin{itemize}
 \item $\mathcal{M}$ is a manifold;
 \item  $\circ:\cT_\mathcal{M}\otimes \cT_\mathcal{M}\to \cT_\mathcal{M}$ is a multiplication operation on the tangent sheaf;  
 \item $\cT_\mathcal{M}^f\subset \cT_\mathcal{M}$ is the subsheaf of flat vector fields; 
 \item $g:S^2(\cT_\mathcal{M})\to \cO_\mathcal{M}$ is the non-degenerate symmetric quadratic form with $\cO_\mathcal{M}$ a sheaf of holomorphic functions on $\mathcal{M}$. 
 \end{itemize}

\smallskip 
 
The main additional structure gluing together this data is given by a family of (local) potentials $\Phi,$ 
being sections of $\cO_\mathcal{M}$, such that 
for any local tangent fields $X,Y,Z$:
\begin{equation}\label{E:!}
g(X\circ Y,Z)=g(X,Y\circ Z)=(XYZ)\Phi.
\end{equation}

If such a structure exists on $(\mathcal{M};\, \circ:\cT_\mathcal{M}\otimes \cT_\mathcal{M}\to \cT_\mathcal{M};\, \cT_\mathcal{M}^f\subset \cT_\mathcal{M}; \ g:S^2(\cT_\mathcal{M})\to \cO_\mathcal{M})$ then it forms a pre-Frobenius manifold.

\begin{dfn}
A pre-Frobenius manifold is called associative if the multiplication $\circ$ is associative. 

A pre-Frobenius manifold is called potential if $C$ admits everywhere locally admits a potential.

A pre-Frobenius manifold  is Frobenius if it is simultaneously potential and associative.    
\end{dfn} 

\smallskip

 \begin{thm}[Th.1.5, \cite{Man99}]
 The pre-Frobenius manifold $(\cM,\circ,\nabla_0)$ is a Frobenius manifold with compatible flat structure if and only if $\nabla_{\lambda}^\cA$ is a pencil of torsionless, flat connections. 
 \end{thm}

\smallskip 
\subsection{Associativity equations}\label{S:WDVV}
\smallskip 

Consider an $F$-manifold $(\mathcal{M},\circ,\nabla_0)$ endowed with a compatible flat structure $\nabla_0.$

Consider the algebra $(\mathscr{A}, \circ)$ with basis $\{e_i\}_{i=1}^n$ and its structural constants $C^s_{ij}$ satisfying the relation \[e_i\circ e_j=C^s_{ij}e_s\] for indexes $i,j,s\in \{1,\cdots, n\}$. This algebra is defined at each point $p\in \mathcal{M}$ on the tangent space to $T_\mathcal{M}$.

If we choose a local flat coordinate system $(x_a)$ on $\cM$ and write the local vector potential $C$ as $C=\sum_cC^c\partial_c$, $X=\partial_a$ and $Y=\partial_b$ then 
 \[\partial_a\circ \partial_b=\sum C_{ab}^c\partial_c, \quad C^c_{ab}=\partial_a\partial_bC^c.\] 
 
 The choice of a flat invariant metric $g$ allows to define $C^a$ as $C^a=\sum_bg^{ba}\partial_b\Phi $. 
 
Let us consider:

\begin{equation}\label{LHS}
    (\partial_a\circ \partial_b)\circ \partial_c=\Bigg( \sum_e C^e_{ab}\partial_e\Bigg)\circ \partial_c=\sum_{ef} C_{ab}^eC^f_{ec}\partial_f.
\end{equation}

\smallskip

\begin{equation}\label{RHS}
\partial_a\circ (\partial_b\circ \partial_c) =\partial_a \circ \sum_f C^f_{bc}\partial_f=(-1)^{a(b+c+e)}\sum_{ef} C_{bc}^fC^e_{ae}\partial_e.
\end{equation}

If $\mathscr{A}$ is associative then
\[ C_{ab}^eC_{ec}^f=C_{ae}^eC_{cb}^f, \]

which amounts to $\partial_a\circ (\partial_b\circ \partial_c) = (\partial_a\circ \partial_b)\circ \partial_c$ i.e. Eq.\ref{LHS}=Eq.\ref{RHS}. 

Given that $C^c_{ab}:=\sum_e C_{abe} g^{ec}=\sum_e \Phi_{abe} g^{ec}$, where $(g^{ec}):=(g_{ec})^{-1}$ and $\Phi_{abe}=\partial_a\partial_b\partial_e\Phi$ we can rewrite the associativity relation for $\Phi$ as
 
 \[\forall a,b,c,d:\quad \sum_{ef}\Phi_{abe}g^{ef}\Phi_{fcd}=(-1)^{a(b+c)}\sum\Phi_{bce}g^{ef}\Phi_{fad},\]
 which corresponds to the WDVV equation. Given the flat identity denoted $e=\partial_0$, then the equation 
reduces to $\Phi_{0ab}=g_{ab}$.

\smallskip

\section{The WDVV equation and symmetric cones }
We proceed as follows.
\begin{enumerate}
\item We  prove the existence of totally geodesic submanifolds in Vinberg cones. 
\item We prove that those totally geodesic submanifolds carry Frobenius structures (i.e. satisfy the WDVV equation). 
\end{enumerate}
\subsection{A first insight}
To give an first  intuition to the reader of why should totally geodesic submanifolds in Vinberg cones exist, it is enough to go back to a theorem of Rothaus~\cite{Rot}. According to \cite{Rot}: {\it if  $\Omega$ is a cone of rank $k$ and of dimension $n$ then every cone $\Omega$ is birationally biregularly equivalent to the direct product of $k$ half lines and an Euclidean space.} The number of half lines appearing coincides with the rank of the Vinberg cone $\Omega$. 

\subsection{Totally geodesic submanifolds and Lie triple systems}
We will proceed to an algebraic investigation of the existence of totally geodesic submanifolds in $\Omega$. Using our method, it is possible to compute explicitly all the properties of those manifolds.

\begin{dfn}
Let $\Omega$ be a noncompact symmetric space.
A totally geodesic $r$-dimensional submanifold of $\Omega$ isometric to $\mathbb{R}^r $ is called an $r$-flat. If $r$ is the maximal natural number $r$ for which an $r$-flat exists then  it is a maximal flat.
\end{dfn}

Let $F$ denote an $r$-flat submanifold in an $n$-dimensional Vinberg cone $\Omega$.
We first outline an algebraic description of $F$.

\smallskip

\begin{thm}[Thm IV.4.2. and Thm IV.7.2 of \cite{Hel}]\label{T:Hel}
\begin{enumerate}~
 
    \item The curvature tensor $R$ evaluated at $T_e\Omega$ is given by 
\[R(X,Y)Z=-[[X,Y],Z],\, \text{for}\quad  X,Y,Z\, \in\,  T_e\Omega. \]

\item   The totally geodesic space $F\subset \Omega$ has the form 
$$F=\exp{\frak{a}}\cdot e,$$ where $\frak{a}\subseteq \ft$ is a Lie triple system i.e. $[[\frak{a},\frak{a}],\frak{a}]\subseteq \frak{a}$.
\item Totally geodesic submanifolds through $e$ are of the form $ \exp{\frak{a}}\cdot e$ where $\frak{a}\subseteq \ft $ is a Lie triple system. 

\end{enumerate}
\end{thm}

By construction, we know that the sectional curvature restricted to $F$ equals zero. If $\frak{a}\subseteq \ft$ is a maximal abelian subspace of dimension $r$ then,  $F=\exp{\frak{a}}\cdot e$ is a maximal flat in $\Omega$.

\begin{rem}
    Given that $\fg$ are semisimple Lie algebras a possible class of Lie subalgebras $\frak{a}$ that satisfy the necessary requirements for producing a maximal flat are the Cartan subalgebras of $\fg$. They are maximal abelian subalgebras of $\fg$ and form a Lie triple system. 
\end{rem}

\begin{lem}
    Suppose that there exists a totally geodesic $F$ submanifold immersed in $\Omega$. Then, the Lie subalgebra $\frak{a}$ of $\ft$ attached to the tangent space $T_eF$ to $F$ is Lie associative.
\end{lem}
\begin{proof}
   By  Thm IV 4.2 in \cite{Hel}, the curvature tensor $R_0$ evaluated at $T_e\Omega$ is given by 
\begin{equation}\label{E:courb}
    R_0(X,Y)Z=-[[X,Y],Z], \quad X,Y,Z \in T_e\Omega.
\end{equation}
$F$ being totally geodesic it is flat. So, restricting attention onto $F$, the left hand side of Eq.\ref{E:courb} is 0. Thus, we have
$0=-[[X,Y],Z]$,  where $X,Y,Z \in T_eF$. Now, by Jacobi's identity: 
\[[X,[Y,Z]]+[Y,[Z,X]]+[Z,[X,Y]]=0.\] Rewriting it one gets $[[Y,Z],X]+[[Z,X],Y]=-[[X,Y],Z]$.
So, by Eq.\ref{E:courb} for $X,Y,Z \in T_eF$ one has
$[[Y,Z],X]=[Y,[Z,X]]$. Thus, the Lie associativity is satisfied on $F$.  
\end{proof}

\subsection{The case of space-like Lagrangian--Grassmanians cones}
We prove algebraically that in $n$-dimensional space-like Lagrangian--Grassmanians cones there exist totally geodesic immersed submanifolds being maximal flats in $\Omega$.

According to Thm.~\ref{T:Hel} totally geodesic submanifold are found if one has a Lie triple system. So, in particular we can take the example of Cartan subalgebras which form a Lie triple sytem. A generalisation is mentioned as  Cor.~\ref{C:totallygeodesic}.

\begin{prop}\label{P:Matrix}
Consider the $n$-dimensional irreducible Vinberg cones of Lagrangian--Grassmanian type. Then, for each of those cones there exists an $n-1$-dimensional totally geodesic submanifold given by $$F=\exp{\tilde{\frak{a}}}\cdot e,$$ such that ${\tilde{\frak{a}}}$ is a Cartan subalgebra of $\frak{gl}_n(\K)$ given by 
${\tilde{\frak{a}}}=\lambda I_n\oplus \frak{a}$, $\lambda\in\K$ where $\frak{a}$ is formed by diagonal matrices of null trace.  
\begin{enumerate}
    \item  If $\K=\R$, $\frak{a}$  is given by all diagonal matrices with real diagonal entries and such that the trace is 0. 
    \item  If $\K=\C$, $\frak{a}$  is given by all diagonal matrices with diagonal entries $a+\imath b$ and such that the trace is 0. 

    \item If $\K=\hH$, $\frak{a}$ is given by all diagonal matrices of $\bigg\{\begin{pmatrix}
        X& -\overline{Y}\\
        Y & \overline{X}
    \end{pmatrix}\, |\, \Re Tr X=0. \bigg\}$ 
    
    \item If $\K=\oO$, $\frak{a}$ is given by all diagonal $(3\times 3)$ matrices with diagonal entries $a+\imath b$ and of the diagonal matrices of a Cartan subalgebra of $\frak{g}_{2}$.
\end{enumerate}
\end{prop}
\begin{proof}
The noncompact symmetric spaces considered are of the type: $SL_n(\K)/K\rtimes \K^*$, where $K$ is the maximal compact subgroup at $e$. Thm IV.7.2 of \cite{Hel} implies that totally geodesic submanifolds through $e$ are of the form $exp{\frak{a}}\cdot e$, where $\frak{a}\subset \ft$ is a Lie triple system. Given that we are working with semi-simple Lie algebras we can investigate in particular the existence of their Cartan subalgebras which here will be maximal abelian subalgebras. The Cartan involution is given by $ X\mapsto -X^t$.
\smallskip 
One can prove that for the real division algebras $\K$, the Cartan subalgebras of $\frak{sl}_n(\K)$ are  diagonal matrices with null trace. We discuss this statement in detail, according to $\K$. 

\begin{enumerate}
    \item If $\K=\R$, then Lie algebra attached to the symmetric space has the Cartan decomposition $\frak{sl}_n(\R)=\frak{so}_n\oplus sym_0(n)$ where $\frak{s}ym_0(n)$ denotes the set of symmetric matrices of trace 0 with entries in $\R$. One can check that the maximal abelian subspace $\frak{a}$ of $sym_0(n)$ is given by the set of diagonal matrices of null trace.
The maximal flat is thus given by  
\[F=\exp{\frak{a}}\cdot e=\{Diag(\lambda_1,\cdots, \lambda_n): \,  \lambda_i=exp(t_i)\in\, \R,\,  \prod_{i=1}^n \lambda_i=1\}.\]

  \item If $\K=\C$, $\frak{a}$ is given by the diagonal matrices with complex entries. The maximal flat is 
\[F=\exp{\frak{a}}\cdot e=\{Diag(\lambda_1,\cdots, \lambda_n): \,  \lambda_i=\exp{a_i+\imath b_i}\in\, \C,\, \prod_{i=1}^n \lambda_i=1.\}\]

  \item If $\K=\hH$, $\frak{sl}_n(\hH)$ is given by diagonal matrices in $\bigg\{\begin{pmatrix}
        X& -\overline{Y}\\
        Y & \overline{X}
    \end{pmatrix}\, |\, \Re Tr X=0 \bigg\}.$ To define $\frak{a}$, one takes only the diagonal matrices in that set.

  \item If $\K=\oO$, $\frak{a}$ is given by all  diagonal $(3\times 3)$ matrices with diagonal entries $a+\imath b$ and of the diagonal matrices of a Cartan subalgebra of $\frak{g}_{2}$.
\end{enumerate}
\end{proof}

\begin{prop}\label{P:F-alg}
Consider a Vinberg cone of space-like Lagrangian--Grassmanian type. Then, 
at any point of the maximal flat, the tangent space to it carries the structure of a Frobenius algebra. 
\end{prop}
\begin{proof}

Let us discuss the case of  Vinberg cones of Lagrangian--Grassmanian type. Consider the maximal flat $F$, given by $F=\exp{\tilde{\frak{a}}}\cdot e$ where ${\tilde{\frak{a}}}=\lambda I_n\oplus \frak{a}$, $\lambda\in\K$ is a Cartan subalgebra of $\frak{gl}_n(\K)$ and where $\frak{a}$ are diagonal matrices of null trace. The set of diagonal matrices form an associative, commutative and unital algebra. Therefore, we have a Frobenius algebra structure on the tangent space to the $n-1$-flat. The tangent space to $F$ is given by $\tilde{\frak{a}}\subseteq \ft$. On $\ft$ there exists a bilinear symmetric form given by the Killing form. In particular, we have $\langle X,Y\rangle=tr(XY)$, where $X,Y\in T_e(G/K)$, which satisfies associativity by Lem.~\ref{L:Ass}. The Cartan subalgebra inherits this bilinear form. Thus, we have a Frobenius algebra.  
\end{proof}
\begin{lem}
   The $n$-dimensional Vinberg cones (1)--(3) are associated  with the Weyl chambers of type $A_n$.
\end{lem}
\begin{proof}
    For $SL_n(\R)/SO_n$ consider $H\in Diag(t_1,\cdots t_n)\in \frak{a}$. 
    One gets that \[(ad H)E_{ij}=[H,E_{ij}]=(t_i-t_j)E_{ij}.\] So, one has $n(n-1)$ non-zero roots. In particular, we have the following splitting
    \[\frak{sl}_n(\R)=\frak{a}+\sum_{i\neq j}\R\cdot E_{ij}.\]

  For $\C$, a similar situation occurs. More abstractly we can write:  \[\frak{sl}_n(\C)=\frak{a}+\oplus \bigg(\bigoplus_{i\neq j}\fg_{\lambda_i-\lambda_j},)\]
  where $\fg_{\lambda_i-\lambda_j}=Spac_{\C}(e_{ij})$ and $e_{ij}$ represents the basis vector in the $i$-th row and $j$-th column. For $\hH$ and $\oO$ a similar argument can be carried out.
\end{proof}

\begin{lem}
Let $\Omega$ be Vinberg cones of type (1)--(4).
A Weyl chamber is isomorphic to $\exp{\frak{a}^+}\in V$, where 
\begin{equation}\label{E:a+}
    \frak{a}^+:=\{Diag(t_1,\cdots, t_n):\, \sum_{i=1}^n t_i=0,\, t_1<\cdots <t_n\}.
\end{equation}
\end{lem}
\begin{proof}
A Weyl chamber is isomorphic to an open Euclidean cone in $\frak{a}$. This chamber is given by Eq.\ref{E:a+}.
So, in the Vinberg cone $\Omega$ a Weyl chamber is isomorphic to $\exp{\frak{a}^+}\in V$.
\end{proof}

\subsection{The case of Lorentzian cones}

\begin{prop}\label{P:Lorentz}
 Consider the irreducible Lorentzian cone $\Lambda_n$. Then, there exists a totally geodesic submanifold $\mathscr{H}$ in $\Lambda_n$. This totally geodesic space is given by the maximal flat $\exp{\frak{a}}\cdot e$, where $\frak{a}$ is a maximal abelian subalgebra of the Lie algebra $\frak{o}(1,n-1)$. 
 \end{prop}

\begin{proof}
The Lorentzian cone is related to pseudo-orthogonal Lie algebras. It is described in terms of quotients of Lie groups by $O(p,q)/O(p)\times O(q)$, with $p\geq q$, where $p=n-1$ and $q=1$.

To  the corresponding Lie algebra, there exists a non-empty maximal abelian subalgebra.
Consider the matrix corresponding to an element of $\frak{o}(p,q)$:
\begin{center}
\includegraphics[scale=0.2]{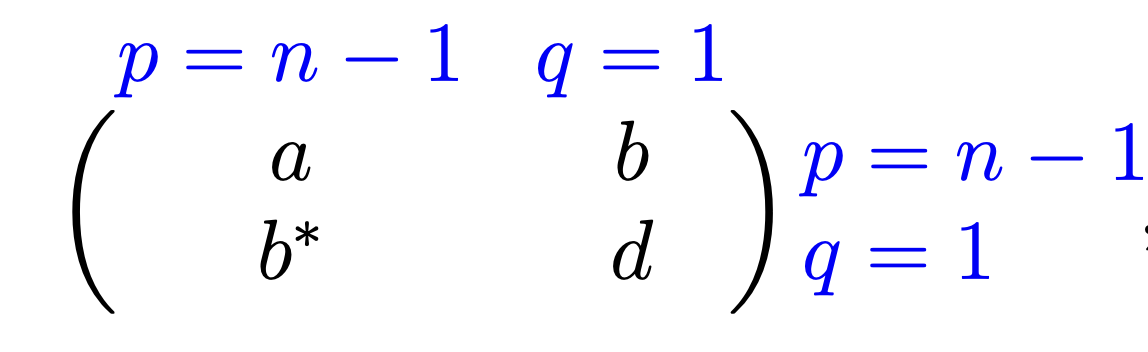}
\end{center}

where entries are real and $a,d$ are skew symmetric. 
The blue numbers correspond to numbers of rows/columns. We take $\textgoth{k}$ as matrices with $b=0 $; $\ft$ are matrices with $a=0$ and $d=0$. It has been shown that for $\frak{o}(p,q)$, one has a class of maximal abelian subalgebra which are formed by orthogonally decomposable matrices (these are the Cartan subalgebras). All matrices in that set can be simultaneously represented by block diagonal matrices, having the same decomposition patterns. 
A second type of maximal abelian subalgebra  occurring has a matrix representation as follows: \begin{equation}\label{E:typ2}
    \begin{pmatrix}
   0 & \alpha & \cdots &0 & \\
   \\  \cdots & & \cdots & -\alpha^T\\
   \\    0 & & \cdots & 0\\
     &0 & &\cdots & 0\\
\end{pmatrix}_{n\times n}
\text{for some vector} \, \alpha.
\end{equation}

We construct the maximal flats via  Thm. \ref{T:Hel}. The flat is a non-empty subspace given that the maximal abelian subalgebra.  So, there exists a non-empty totally geodesic submanifold $\cH$ in $\Lambda_n.$ 
\end{proof}
\smallskip 

The Lorentzian cone $\Lambda_n$ is of the Anti-de-Sitter type. Anti-de-Sitter spaces are Lorentzian manifolds with negative constant sectional curvature.

\begin{lem}
Consider $\mathscr{H}$ the totally geodesic submanifold in $\Lambda_n$. Then, the tangent space at at any point of $\mathscr{H}$ carries the structure of a (non unital) Frobenius algebra. 
\end{lem}
\begin{proof}
We have a maximal abelian subalgebra $\frak{a}\subset \ft$ which is associated to $\mathscr{H}$, via $\exp{\frak{a}}\cdot e$ . This defines the ``flat''. Taking $\frak{a}$ to be represented by block diagonal matrices having same decomposition pattern, one can prove that for the standard matrix multiplication this forms this forms an associative, commutative algebra. Given that $\frak{a}$ inherits a symmetric bilinear form from $\ft$ satisfying Lem.~\ref{L:Ass} it is a Frobenius algebra. Associativity and commutativity also holds for the class of matrices of type given in formula~\ref{E:typ2} as well as the symmetric bilinear form satisfying Lem.~\ref{L:Ass}. 

\end{proof}

\subsection{Main statements}
\subsection{New sources of $F$-manifolds}
We prove first that on $F$, the multiplication $\circ$ on the tangent sheaf admits a compatible flat structure.

\begin{prop}\label{P:compatible}
Let $\Omega$ be a Vinberg cone. 
    Consider $F$ a totally geodesic immersed submanifold of $\Omega$. Then, for any vector fields $X,Y$ in $\cT_F$ there exists a local vector potential $\mathscr{C}$ such that we have the following compatibility relation 
    \begin{equation}\label{E:Crochet}
        X\circ Y=[X,[Y,\mathscr{C}]].
    \end{equation}
\end{prop}

\begin{proof}
By \cite{Hel}, Thm IV 4.2 one has the following:
$$R_0(Y,Z)X=[X,[Y,Z]],$$ where $X,Y,Z \in T_e\Omega$.
The submanifold $F$ being totally geodesic, the scalar curvature vanishes on $F$. So,  
the curvature tensor is 0 and the right hand side of Eq.\ref{E:Crochet} is 
$$[X,[Y,Z]]=0.$$ 

Consider the left hand side of Eq.~\ref{E:Crochet}. The multiplication operation in the algebra is given by: $$(X\circ Y)^i=\Gamma_{jk}^iX^jY^k,$$ where $\Gamma_{jk}^i=g^{li}\partial_{jkl}\Phi$. The scalar curvature vanishes if and only if $C_{ijk}=\partial_{ijl}\Phi$ vanishes.
So, $X\circ Y =0$ and thus $X\circ Y=[X,[Y,\mathscr{C}]]$.    
\end{proof}

\begin{thm}\label{T:F-man}
    Let $\Omega$ be a Vinberg cone. Consider $F$ a totally geodesic immersed submanifold of $\Omega$ given by $F=\exp{\frak{a}}\cdot e$, where $\frak{a}\subset \ft$ is a Lie triple system. Then, the flat structure on $F$ is compatible with $\circ$ and $(F,\circ)$ is an $F$-manifold. 
\end{thm}

\begin{proof}
    One applies directly the statement from \cite{Hel} (Theorem~\label{T:Hel}) ensuring that $F$ is totally geodesic iff  $\frak{a}\subset \ft$ is a Lie triple system.
     We have shown in Prop.~\ref{P:compatible} that one has a flat structure compatible with $\circ$ on $F$. Thus, by Manin [Man] if the multiplication $\circ$ on the tangent sheaf admits a compatible flat structure it forms an $F$-manifold. 
        
\end{proof}

\
\begin{cor}\label{C:totallygeodesic}
    Let $\Omega$ be a Vinberg cone. Then, $F=\exp{\frak{a}}\cdot e$ in $\Omega$ is an $F$-manifold iff $\frak{a}$ is a Lie triple system.
\end{cor}

\begin{proof}
Any Vinberg cone is a linear combination of irreducible Vinberg cones.
Each of those irreducible cones contains a non-empty totally geodesic submanifold (Prop.\ref{P:Lorentz} and Prop.\ref{P:Matrix}). The Cartesian product of those totally geodesic submanifolds gives a totally geodesic submanifold. So, our statement on the existence of a totally geodesic submanifold in a Vinberg cone holds given {\it any} Vinberg cone defined on any algebra (for example a Clifford algebra) combining the real division algebras.  

\end{proof}

\begin{lem}
Let $F=\exp{\frak{a}}$, where $\frak{a}$ is Lie triple system. Then,  $(F,\circ,\nabla_0)$ is an $F$-manifold with compatible flat structure iff $\nabla_\lambda^\mathcal{A}$ is a pencil of torsionless flat connections.    
\end{lem}
\begin{proof}
    This follows from Thm.~\ref{T:F-man} and Prop.\ref{P:F}.
\end{proof} 

\subsection{New sources of Frobenius manifolds}

\begin{prop}\label{P:pre-Fro}
Vinberg cones are equipped with the structure of a potential pre-Frobenius manifold. 
\end{prop}

\begin{proof}
Let $\Omega$ be a Vinberg cone. 

\begin{enumerate}
    \item By Lem.~\ref{L:affst}, $\Omega$ has an affine flat structure. Given that it is a Hessian manifold the flat connection is a Levi--Civita connection $\nabla_0$. 
    \item Vinberg cones comes equipped with the data $(\Omega,g,C,\circ)$, where:
\end{enumerate}

\begin{itemize}
    \item $g$ is a compatible Riemannian (Hessian) metric. Locally we have: $g_{ij}=\partial_i\partial_j\Phi$, where $\Phi$ is a potential function given by the Koszul--Vinberg function given in Def.\ref{D:KV}.
     \smallskip
    \item $C$ is a rank three symmetric tensor. In the present case given that $\Omega$ is being Hessian , $C$ is given by $C_{ijk}=\partial_i\partial_j\partial_k\Phi$.
    \smallskip
    \item $``\circ"$ is a multiplication operation on the tangent bundle, defined in Prop. \ref{P:F-alg}. The subalgebra associated to the locus of $(n-1)$-flats is a Frobenius algebra, by Prop.\ref{P:FA}. 
\end{itemize}

By Prop.\ref{P:F-alg} and Lem.\ref{L:Ass}, we have $C(X,Y,Z)=g(X\circ Y, Z)=g(X,Y\circ Z)$. 

This gives a pre-Frobenius manifold. One can add that we have a potential pre-Frobenius manifold. Indeed, $C(X,Y,Z)$ admits everywhere locally a potential. 
\end{proof}
\begin{thm}\label{C:VinFrob}
Consider a Vinberg cone $\Omega$. Then, the manifold $F=\exp{\frak{a}}\cdot e$, where $\frak{a}\subset \ft$ is Lie triple system, forms a Frobenius manifold immersed in $\Omega$.
\end{thm}

\begin{proof}
Let $F$ be maximal flat in $\Omega$; we think of $F$ as embedded in $\Omega$ for the following paragraphs. 

\smallskip 

1) By Prop.~\ref{P:pre-Fro} the Vinberg cone is pre-Frobenius. It is equipped with: 

\begin{itemize}
    \item $\nabla_0$, an affine flat structure (cf. Lem.~\ref{L:affst});
    \item $g$, a compatible Riemannian metric; 
    \item $C$, a rank three symmetric tensor; 
      \item $``\circ"$, a multiplication operation on the tangent bundle (see Prop. \ref{P:FA}.)
\end{itemize}

Moreover, one has a compatibility condition: $g(X\circ Y,Z)=g(X,Y\circ Z)=C(X,Y,Z)$, where $X, Y,Z\in \cT_\Omega$ is satisfied.

Let us now discuss the proper embedding $i:F\hookrightarrow V$. $F$ is isometric to an Eucldean space.

\smallskip 

2) By Lem.~\ref{L:affst}, $\Omega$ has an affine flat structure. The embedding $i:F \to V$ induces a Levi--Civita connection on $F$. It is the Levi--Civita connection for the pullback metric $g^F:=i^*g.$

\smallskip 

3) One verifies that the potentiality axiom holds on $F$. Indeed, as for $\Omega$, the rank 3 symmetric tensor $C$ admits everywhere locally a potential function. The local potential $\Phi(=\ln\chi)$ is such that for any flat local tangent fields  one has: $\partial_{abc}\Phi=C_{abc}$. 

\smallskip 

4)  By~\ref{P:F-alg} the tangent space to $F$ forms a Frobenius algebra.
 This is a strong statement since it proves that the associativity axiom for $\circ$ i.e $(X\circ Y)\circ Z=X\circ (Y\circ Z)$ for $X,Y,Z\in \cT_V$.

 At this stage of the proof all axioms to have a Forbenius manifold are satisfied.  So, a maximal flat i.e a totally geodesic submanifold of maximal dimension, in a Vinberg cone froms a Frobenius manifold.  

A different argument can be added. 
In virtue of Th. 1.5 in~\cite{Man99} a pre-Frobenius manifold M is Forbenius iff 
the pencil of connections $\nabla_{\lambda,X}(Y)=\nabla_{0,X}(Y)+\lambda (X\circ Y)$ for $X,Y\in \cT_V$ and $\lambda\in \R$ is flat.
The maximal flat is totally geodesic. So, the pencil of connections $\nabla_{\lambda}^\cA$ is flat. This allows a direct conclusion. 
\end{proof}

\begin{lem}
 Every Frobenius submanifold $F$ of $\Omega$ is
necessarily a $G$-translate of $F$, where $G$ acts by isometries on $\Omega$.
   \end{lem}

\section{Conclusion and future work}
We have thus proved in this paper an interesting relation between the WDVV PDE and strictly convex symmetric cones, leading to intriguing questions relating the 2D (TFT) \cite{Du96} and those cones.
This is achieved using geometric and algebraic methods (using Cartan's classification of symmetric spaces) on the one hand side; and Calabi's investigation on the Monge--Amp\`ere equation. This allows for instance to generalise a result of \cite{FKN, Mokh95} on symplectic Monge--Ampere equations of Hirota type. It as well allows an expression of this idea in the realm of geometry and algebra. 

Applications of this result can be achieved for instance in information geometry \cite{CoMa,Ch82}. This is done by using the fact that 
to any point $x\in \Omega$ one may attach a probability measure on the dual cone. This forms a main object in information geometry and statistics, since the {\it KV characteristic function} corresponds to the density of a probability measure:\[\forall x \mapsto p_x(x^*)=\frac{\exp(-\langle x^*,x\rangle)}{\chi(x)}.\]
 
To conclude, results of this paper reveal an enlargement of the classification of sources of Frobenius manifolds. In particular, it is shown that there exist Frobenius manifolds defined over linear combinations of real division algebras of finite dimension and as well that there exist Frobenius manifolds in the context of pseudo-Riemannian manifolds: the Lorentz manifolds of Anti-De-Sitter type.

\end{document}